\documentclass[12pt]{amsart}
\usepackage{amsmath,amssymb,amsbsy,amsfonts,amsthm,latexsym,
                        amsopn,amstext,amsxtra,euscript,amscd,mathrsfs,color,bm}
                        
\usepackage{float} 
\usepackage[english]{babel}
\usepackage{mathtools}
\usepackage{todonotes}
\usepackage{url}
\usepackage[colorlinks,linkcolor=blue,anchorcolor=blue,citecolor=blue,backref=page]{hyperref}
\usepackage{enumitem}

\def\le{\leqslant}
\def\ge{\geqslant}

\def\cc#1{\textcolor{red}{#1}}

\usepackage[norefs,nocites]{refcheck}

\begin{document}
%
    
\newtheorem{theorem}{Theorem}
\newtheorem{lemma}[theorem]{Lemma}
\newtheorem{example}[theorem]{Example}
\newtheorem{algol}{Algorithm}
\newtheorem{corollary}[theorem]{Corollary}
\newtheorem{prop}[theorem]{Proposition}
\newtheorem{proposition}[theorem]{Proposition}
\newtheorem{problem}[theorem]{Problem}
\newtheorem{conj}[theorem]{Conjecture}
\newtheorem{cor}[theorem]{Corollary}

\theoremstyle{remark}
\newtheorem{definition}[theorem]{Definition}
\newtheorem{question}[theorem]{Question}
\newtheorem{remark}[theorem]{Remark}
\newtheorem*{acknowledgement}{Acknowledgements}

\numberwithin{equation}{section}
\numberwithin{theorem}{section}
\numberwithin{table}{section}
\numberwithin{figure}{section}

\allowdisplaybreaks

\definecolor{olive}{rgb}{0.3, 0.4, .1}
\definecolor{dgreen}{rgb}{0.,0.5,0.}

\def\cc#1{\textcolor{red}{#1}} 

\def\Cde{C_{d}}
\def\tCde{\widetilde{C}_{d}}

\definecolor{dgreen}{rgb}{0.,0.6,0.}
\def\tgreen#1{\begin{color}{dgreen}{\it{#1}}\end{color}}
\def\tblue#1{\begin{color}{blue}{\it{#1}}\end{color}}
\def\tred#1{\begin{color}{red}#1\end{color}}
\def\tmagenta#1{\begin{color}{magenta}{\it{#1}}\end{color}}
\def\tNavyBlue#1{\begin{color}{NavyBlue}{\it{#1}}\end{color}}
\def\tMaroon#1{\begin{color}{Maroon}{\it{#1}}\end{color}}

\def\ndiv{\nmid}
\def \balpha{\bm{\alpha}}
\def \bbeta{\bm{\beta}}
\def \bgamma{\bm{\gamma}}
\def \bdelta{\bm{\delta}}
\def \blambda{\bm{\lambda}}
\def \bchi{\bm{\chi}}
\def \bphi{\bm{\varphi}}
\def \bpsi{\bm{\psi}}
\def \bnu{\bm{\nu}}
\def \bomega{\bm{\omega}}

\def\td{\widetilde d}
\def \te{\widetilde e} 
\def\talpha{\widetilde \alpha}
\def \tbeta{\widetilde \beta} 

\def \tcA{\widetilde {\cA}}
\def  \tcB{\widetilde{\cB}}

\def\vh{\vec{h}} 
\def\vj{\vec{j}} 

\def\vk{\vec{k}} 
\def\vl{\vec{l}} 

\def\vu{\vec{u}} 
\def\vv{\vec{v}} 

\def\vx{\vec{x}} 
\def\vy{\vec{y}}


 \def\mand{\qquad\mbox{and}\qquad}

\def\cA{{\mathcal A}}
\def\cB{{\mathcal B}}
\def\cC{{\mathcal C}}
\def\cD{{\mathcal D}}
\def\cE{{\mathcal E}}
\def\cF{{\mathcal F}}
\def\cG{{\mathcal G}}
\def\cH{{\mathcal H}}
\def\cI{{\mathcal I}}
\def\cJ{{\mathcal J}}
\def\cK{{\mathcal K}}
\def\cL{{\mathcal L}}
\def\cM{{\mathcal M}}
\def\cN{{\mathcal N}}
\def\cO{{\mathcal O}}
\def\cP{{\mathcal P}}
\def\cQ{{\mathcal Q}}
\def\cR{{\mathcal R}}
\def\cS{{\mathcal S}}
\def\cT{{\mathcal T}}
\def\cU{{\mathcal U}}
\def\cV{{\mathcal V}}
\def\cW{{\mathcal W}}
\def\cX{{\mathcal X}}
\def\cY{{\mathcal Y}}
\def\cZ{{\mathcal Z}}

\def\C{\mathbb{C}}
\def\F{\mathbb{F}}
\def\K{\mathbb{K}}
\def\Z{\mathbb{Z}}
\def\R{\mathbb{R}}
\def\Q{\mathbb{Q}}
\def\N{\mathbb{N}}
\def\L{\mathbb{L}}
\def\M{\textsf{M}}
\def\U{\mathbb{U}}
\def\P{\mathbb{P}}
\def\A{\mathbb{A}}
\def\fp{\mathfrak{p}}
\def\fq{\mathfrak{q}}
\def\n{\mathfrak{n}}
\def\X{\mathcal{X}}
\def\x{\textrm{\bf x}}
\def\w{\textrm{\bf w}}
\def\ovQ{\overline{\Q}}
\def \Kab{\K^{\mathrm{ab}}}
\def \Qab{\Q^{\mathrm{ab}}}
\def \Qtr{\Q^{\mathrm{tr}}}
\def \Kc{\K^{\mathrm{c}}}
\def \Qc{\Q^{\mathrm{c}}}
\def\ZK{\Z_\K}
\def\ZKS{\Z_{\K,\cS}}
\def\ZKSf{\Z_{\K,\cS_f}}
\def\RSf{R_{\cS_{f}}}
\def\RTf{R_{\cT_{f}}}

\def\S{\mathcal{S}}
\def\vec#1{\mathbf{#1}}
\def\ov#1{{\overline{#1}}}
\def\Sp{{\operatorname{S}}}
\def\Gm{\G_{\textup{m}}}
\def\fA{{\mathfrak A}}
\def\fB{{\mathfrak B}}
\def\fM{{\mathfrak M}}

\def \brho{\bm{\rho}}
\def \btau{\bm{\tau}}

\def\house#1{{%
    \setbox0=\hbox{$#1$}
    \vrule height \dimexpr\ht0+1.4pt width .5pt depth \dp0\relax
    \vrule height \dimexpr\ht0+1.4pt width \dimexpr\wd0+2pt depth \dimexpr-\ht0-1pt\relax
    \llap{$#1$\kern1pt}
    \vrule height \dimexpr\ht0+1.4pt width .5pt depth \dp0\relax}}


\newenvironment{notation}[0]{%
  \begin{list}%
    {}%
    {\setlength{\itemindent}{0pt}
     \setlength{\labelwidth}{1\parindent}
     \setlength{\labelsep}{\parindent}
     \setlength{\leftmargin}{2\parindent}
     \setlength{\itemsep}{0pt}
     }%
   }%
  {\end{list}}

\newenvironment{parts}[0]{%
  \begin{list}{}%
    {\setlength{\itemindent}{0pt}
     \setlength{\labelwidth}{1.5\parindent}
     \setlength{\labelsep}{.5\parindent}
     \setlength{\leftmargin}{2\parindent}
     \setlength{\itemsep}{0pt}
     }%
   }%
  {\end{list}}
\newcommand{\Part}[1]{\item[\upshape#1]}

\def\Case#1#2{%
\smallskip\paragraph{\textbf{\boldmath Case #1: #2.}}\hfil\break\ignorespaces}

\def\Subcase#1#2{%
\smallskip\paragraph{\textit{\boldmath Subcase #1: #2.}}\hfil\break\ignorespaces}

\renewcommand{\a}{\alpha}
\renewcommand{\b}{\beta}
\newcommand{\g}{\gamma}
\renewcommand{\d}{\delta}
\newcommand{\e}{\epsilon}
\newcommand{\f}{\varphi}
\newcommand{\fhat}{\hat\varphi}
\newcommand{\bfphi}{{\boldsymbol{\f}}}
\renewcommand{\l}{\lambda}
\renewcommand{\k}{\kappa}
\newcommand{\lhat}{\hat\lambda}
\newcommand{\bfmu}{{\boldsymbol{\mu}}}
\renewcommand{\o}{\omega}
\renewcommand{\r}{\rho}
\newcommand{\rbar}{{\ov\rho}}
\newcommand{\s}{\sigma}
\newcommand{\sbar}{{\ov\sigma}}
\renewcommand{\t}{\tau}
\newcommand{\z}{\zeta}


\newcommand{\ga}{{\mathfrak{a}}}
\newcommand{\gb}{{\mathfrak{b}}}
\newcommand{\gn}{{\mathfrak{n}}}
\newcommand{\gp}{{\mathfrak{p}}}
\newcommand{\gP}{{\mathfrak{P}}}
\newcommand{\gq}{{\mathfrak{q}}}

\newcommand{\Abar}{{\ov A}}
\newcommand{\Ebar}{{\ov E}}
\newcommand{\kbar}{{\ov k}}
\newcommand{\Kbar}{{\ov K}}
\newcommand{\Pbar}{{\ov P}}
\newcommand{\Sbar}{{\ov S}}
\newcommand{\Tbar}{{\ov T}}
\newcommand{\gbar}{{\ov\gamma}}
\newcommand{\lbar}{{\ov\lambda}}
\newcommand{\ybar}{{\ov y}}
\newcommand{\phibar}{{\ov\f}}

\newcommand{\Acal}{{\mathcal A}}
\newcommand{\Bcal}{{\mathcal B}}
\newcommand{\Ccal}{{\mathcal C}}
\newcommand{\Dcal}{{\mathcal D}}
\newcommand{\Ecal}{{\mathcal E}}
\newcommand{\Fcal}{{\mathcal F}}
\newcommand{\Gcal}{{\mathcal G}}
\newcommand{\Hcal}{{\mathcal H}}
\newcommand{\Ical}{{\mathcal I}}
\newcommand{\Jcal}{{\mathcal J}}
\newcommand{\Kcal}{{\mathcal K}}
\newcommand{\Lcal}{{\mathcal L}}
\newcommand{\Mcal}{{\mathcal M}}
\newcommand{\Ncal}{{\mathcal N}}
\newcommand{\Ocal}{{\mathcal O}}
\newcommand{\Pcal}{{\mathcal P}}
\newcommand{\Qcal}{{\mathcal Q}}
\newcommand{\Rcal}{{\mathcal R}}
\newcommand{\Scal}{{\mathcal S}}
\newcommand{\Tcal}{{\mathcal T}}
\newcommand{\Ucal}{{\mathcal U}}
\newcommand{\Vcal}{{\mathcal V}}
\newcommand{\Wcal}{{\mathcal W}}
\newcommand{\Xcal}{{\mathcal X}}
\newcommand{\Ycal}{{\mathcal Y}}
\newcommand{\Zcal}{{\mathcal Z}}

\renewcommand{\AA}{\mathbb{A}}
\newcommand{\BB}{\mathbb{B}}
\newcommand{\CC}{\mathbb{C}}
\newcommand{\FF}{\mathbb{F}}
\newcommand{\GG}{\mathbb{G}}
\newcommand{\KK}{\mathbb{K}}
\newcommand{\NN}{\mathbb{N}}
\newcommand{\PP}{\mathbb{P}}
\newcommand{\QQ}{\mathbb{Q}}
\newcommand{\RR}{\mathbb{R}}
\newcommand{\ZZ}{\mathbb{Z}}

\newcommand{\bfa}{{\boldsymbol a}}
\newcommand{\bfb}{{\boldsymbol b}}
\newcommand{\bfc}{{\boldsymbol c}}
\newcommand{\bfd}{{\boldsymbol d}}
\newcommand{\bfe}{{\boldsymbol e}}
\newcommand{\bff}{{\boldsymbol f}}
\newcommand{\bfg}{{\boldsymbol g}}
\newcommand{\bfi}{{\boldsymbol i}}
\newcommand{\bfj}{{\boldsymbol j}}
\newcommand{\bfk}{{\boldsymbol k}}
\newcommand{\bfm}{{\boldsymbol m}}
\newcommand{\bfp}{{\boldsymbol p}}
\newcommand{\bfr}{{\boldsymbol r}}
\newcommand{\bfs}{{\boldsymbol s}}
\newcommand{\bft}{{\boldsymbol t}}
\newcommand{\bfu}{{\boldsymbol u}}
\newcommand{\bfv}{{\boldsymbol v}}
\newcommand{\bfw}{{\boldsymbol w}}
\newcommand{\bfx}{{\boldsymbol x}}
\newcommand{\bfy}{{\boldsymbol y}}
\newcommand{\bfz}{{\boldsymbol z}}
\newcommand{\bfA}{{\boldsymbol A}}
\newcommand{\bfF}{{\boldsymbol F}}
\newcommand{\bfB}{{\boldsymbol B}}
\newcommand{\bfD}{{\boldsymbol D}}
\newcommand{\bfG}{{\boldsymbol G}}
\newcommand{\bfI}{{\boldsymbol I}}
\newcommand{\bfM}{{\boldsymbol M}}
\newcommand{\bfP}{{\boldsymbol P}}
\newcommand{\bfX}{{\boldsymbol X}}
\newcommand{\bfY}{{\boldsymbol Y}}
\newcommand{\bfzero}{{\boldsymbol{0}}}
\newcommand{\bfone}{{\boldsymbol{1}}}

\newcommand{\aff}{{\textup{aff}}}
\newcommand{\Aut}{\operatorname{Aut}}
\newcommand{\Berk}{{\textup{Berk}}}
\newcommand{\Birat}{\operatorname{Birat}}
\newcommand{\characteristic}{\operatorname{char}}
\newcommand{\codim}{\operatorname{codim}}
\newcommand{\Crit}{\operatorname{Crit}}
\newcommand{\critwt}{\operatorname{critwt}} 
\newcommand{\Cycle}{\operatorname{Cycles}}
\newcommand{\diag}{\operatorname{diag}}
\newcommand{\Disc}{\operatorname{Disc}}
\newcommand{\Div}{\operatorname{Div}}
\newcommand{\Dom}{\operatorname{Dom}}
\newcommand{\End}{\operatorname{End}}
\newcommand{\ExtOrbit}{\mathcal{EO}} 
\newcommand{\Fbar}{{\ov{F}}}
\newcommand{\Fix}{\operatorname{Fix}}
\newcommand{\FOD}{\operatorname{FOD}}
\newcommand{\FOM}{\operatorname{FOM}}
\newcommand{\Gal}{\operatorname{Gal}}
\newcommand{\genus}{\operatorname{genus}}
\newcommand{\GITQuot}{/\!/}
\newcommand{\GL}{\operatorname{GL}}
\newcommand{\GR}{\operatorname{\mathcal{G\!R}}}
\newcommand{\Hom}{\operatorname{Hom}}
\newcommand{\Index}{\operatorname{Index}}
\newcommand{\Image}{\operatorname{Image}}
\newcommand{\Isom}{\operatorname{Isom}}
\newcommand{\hhat}{{\hat h}}
\newcommand{\Ker}{{\operatorname{ker}}}
\newcommand{\Ksep}{K^{\textup{sep}}}  
\newcommand{\lcm}{{\operatorname{lcm}}}
\newcommand{\LCM}{{\operatorname{LCM}}}
\newcommand{\Lift}{\operatorname{Lift}}
\newcommand{\limstar}{\lim\nolimits^*}
\newcommand{\limstarn}{\lim_{\hidewidth n\to\infty\hidewidth}{\!}^*{\,}}
\newcommand{\llog}{\log\log}
\newcommand{\logplus}{\log^{\scriptscriptstyle+}}
\newcommand{\Mat}{\operatorname{Mat}}
\newcommand{\maxplus}{\operatornamewithlimits{\textup{max}^{\scriptscriptstyle+}}}
\newcommand{\MOD}[1]{~(\textup{mod}~#1)}
\newcommand{\Mor}{\operatorname{Mor}}
\newcommand{\Moduli}{\mathcal{M}}
\newcommand{\Norm}{{\operatorname{\mathsf{N}}}}
\newcommand{\notdivide}{\nmid}
\newcommand{\normalsubgroup}{\triangleleft}
\newcommand{\NS}{\operatorname{NS}}
\newcommand{\onto}{\twoheadrightarrow}
\newcommand{\ord}{\operatorname{ord}}
\newcommand{\Orbit}{\mathcal{O}}
\newcommand{\Per}{\operatorname{Per}}
\newcommand{\Perp}{\operatorname{Perp}}
\newcommand{\PrePer}{\operatorname{PrePer}}
\newcommand{\PGL}{\operatorname{PGL}}
\newcommand{\Pic}{\operatorname{Pic}}
\newcommand{\Prob}{\operatorname{Prob}}
\newcommand{\Proj}{\operatorname{Proj}}
\newcommand{\Qbar}{{\ov{\QQ}}}
\newcommand{\rank}{\operatorname{rank}}
\newcommand{\Rat}{\operatorname{Rat}}
\newcommand{\Res}{{\operatorname{Res}}}
\newcommand{\Resultant}{\operatorname{Res}}
\renewcommand{\setminus}{\smallsetminus}
\newcommand{\sgn}{\operatorname{sgn}}
\newcommand{\SL}{\operatorname{SL}}
\newcommand{\Span}{\operatorname{Span}}
\newcommand{\Spec}{\operatorname{Spec}}
\renewcommand{\ss}{{\textup{ss}}}
\newcommand{\stab}{{\textup{stab}}}
\newcommand{\Stab}{\operatorname{Stab}}
\newcommand{\Support}{\operatorname{Supp}}
\newcommand{\Sym}{\operatorname{Sym}}  
\newcommand{\tors}{{\textup{tor}}}
\newcommand{\Trace}{\operatorname{Trace}}
\newcommand{\trianglebin}{\mathbin{\triangle}} 
\newcommand{\tr}{{\textup{tr}}} 
\newcommand{\UHP}{{\mathfrak{h}}}    
\newcommand{\Wander}{\operatorname{Wander}}
\newcommand{\<}{\langle}
\renewcommand{\>}{\rangle}

\newcommand{\pmodintext}[1]{~\textup{(mod}~#1\textup{)}}
\newcommand{\ds}{\displaystyle}
\newcommand{\longhookrightarrow}{\lhook\joinrel\longrightarrow}
\newcommand{\longonto}{\relbar\joinrel\twoheadrightarrow}
\newcommand{\SmallMatrix}[1]{%
  \left(\begin{smallmatrix} #1 \end{smallmatrix}\right)}
  
  \def\({\left(}
\def\){\right)}
\def\fl#1{\left\lfloor#1\right\rfloor}
\def\rf#1{\left\lceil#1\right\rceil}


\title[Local--global principle and  additive combinatorics]
{An effective local--global principle  
and additive combinatorics in finite fields}

\author[B.\ Kerr]{Bryce Kerr}
\address{Max Planck Institute for Mathematics, Bonn, Germany}
\email{bryce.kerr89@gmail.com}

\author[J. Mello] {Jorge Mello}
\address{Max Planck Institute for Mathematics, Bonn, Germany}
\email{jbmello@yorku.ca}

\author[I. E. Shparlinski] {Igor E. Shparlinski}
\address{School of Mathematics and Statistics, University of New South Wales, Sydney NSW 2052, Australia}
\email{igor.shparlinski@unsw.edu.au}


\begin{abstract} 
We use recent results about linking the number of zeros on algebraic varieties over $\C$, defined by
polynomials with integer coefficients, and 
on their reductions modulo sufficiently large primes to study congruences with products and 
reciprocals of linear forms. This allows us to make some progress towards a question of 
B.~Murphy, G.~Petridis, O.~Roche-Newton, M.~Rudnev and I.~D.~Shkredov (2019) on an extreme case of 
the  Erd\H{o}s--Szemer{\'e}di conjecture
in finite fields. 
\end{abstract}

\maketitle

\tableofcontents

\section{Introduction}

\subsection{Description of our results}

In this paper we give a new application of a recent result due to D'Andrea, Ostafe, Shparlinski and Sombra~\cite[Theorem~2.1]{DOSS}, which 
establishes an effective link between the number of points on zero dimensional varieties  considered over $\C$ and also
considered in the  field $\F_p$, see Lemma~\ref{lem:TAMS} below. 

In particular, we give sharp upper bounds on the number of solutions to some multiplicative and additive congruences 
modulo primes with variables from sets with small doubling,
see Section~\ref{sec:mult_eq}.

 These  results complement those of Grosu~\cite{Gros}, who has previously applied a similar principle which allows one to study arithmetic in subsets of a finite field by lifting to zero characteristic. The results of Grosu~\cite{Gros} restrict one to consider sets $\cA\subseteq \F_p$ of triple logarithmic size, see~\eqref{eq:GrosuCong} below.  
 Our results (see Section~\ref{sec:mult_eq})  extend the cardinality of the sets considered in some applications (see~\cite[Section~4]{Gros}) to the range $|\cA|\le p^{\delta}$ for some fixed $\delta>0$ which is given explicitly and depends only on the size of $|\cA+\cA|$. We also obtain sharper quantitative bounds for $\delta$ which hold for almost all primes (in the sense of relative asymptotic density).  For example, we prove that if such a set has  small doubling, then its product set is of almost largest possible size, see Theorem~\ref{thm:main24} below. 
 This provides some partial progress towards a question raised by  Murphy, Petridis, Roche-Newton, Rudnev and Shkredov~\cite[Question~2]{MPRRS} which has also been considered by Shkredov~\cite[Corollary~2]{Shk} in a different context and can be considered a mod $p$ variant of a few sums many products estimate due to Elekes and Ruzsa~\cite{ER}, see Section~\ref{sec:appl} for more details.

We note that some arithmetic applications of~\cite[Theorem~2.1]{DOSS} have already been given 
in~\cite{CDOSS, DOSS} (to periods of orbits of some dynamical systems) as well as~\cite{Shp2}
(to torsions of some points on elliptic curves). 

\subsection{General notation}

Throughout this work $\N = \left \{1, 2, \ldots \right \}$ is the set of positive integers. 

For a field $K$, we use $\ov K$ to denote the algebraic closure of $K$. 

For a prime $p$, we use $\F_p$ to denote the finite field of $p$ elements and $\F_p^{*}$ the multiplicative subgroup of $\F_p$.

We freely switch between equations in $\F_p$ and congruences modulo $p$.  

The letters $k$, $\ell,$ $m$ and $n$ (with or without subscripts) are always used to denote positive integers;
the letter~$p$  
(with or without subscripts) is always used to denote a prime.

As usual, for given quantities  $U$ and $V$, the notations $U\ll V$, $V\gg  U$ and
$U=O(V)$ are all equivalent to the statement that the inequality
$|U|\le c V$ holds with some constant $c>0$, which may depend on the 
integer parameter $d$.

Furthermore $V = U^{o(1)} $ means that $\log |V|/\log U\to 0$ as $U\to \infty$. 

We use $|\cS|$ to denote the cardinality of a finite set $\cS$.

For a generic point $\vec{x}\in \R^d$, we write $x_i$ for the $i$-th coordinate of $\x$.  For example, if $\balpha,\vec{h}\in \R^d$ then
$$
\balpha=(\alpha_1,\ldots,\alpha_d) \mand \vec{h}=(h_1,\ldots,h_d).
$$
 Let 
 $$
 \langle \balpha, \vec{h} \rangle=
 \alpha_1 h_1 +\ldots +\alpha_d h_d
 $$ 
 denote the Euclidian inner product and
 $\|\vec{h}\|$  the Euclidean norm of $\vec{h}$.

  For $\balpha \in \R^d$ and $\lambda \in \C$ we let $\lambda \balpha$ denote scalar multiplication
$$\lambda \balpha=(\lambda \alpha_1,\ldots,\lambda \alpha_d).$$

Given a set $\cD\subseteq \R^d$ and $\lambda>0$ we define 
$$\lambda \cD =\{ \lambda x : ~ x\in \cD\}.$$

\section{Main results}

\subsection{Multiplicative equations over sets with small sumsets}
\label{sec:mult_eq}
Let $p$ be prime and for  subsets $\cA,\cB\subseteq \overline\F_p$ and $\lambda\in \overline\F_p$ we define $I_p(\cA,\cB,\lambda)$ by 
\begin{equation}
\label{eq:IpA}
I_p(\cA,\cB,\lambda)=\left| \left\{ (a,b)\in \cA\times \cB   :~  ab = \lambda \right\} \right|,  
\end{equation}
where the equation $ab = \lambda$ is in $\overline\F_p$.

  A  generalised arithmetic progression $\cA$ (defined in any group) is a set of the form 
$$
\cA=\left \{ \alpha_0+\alpha_1h_1+\ldots+\alpha_dh_d :~1\le h_i\le A_i\right \}.
$$
We define the rank of $\cA$ to be $d$ and say $\cA$ is proper if 
$$|\cA|=A_1\ldots A_d.$$

 It is convenient to define 
 \begin{equation}
\label{eq:gamma}
\gamma_s = \frac{1}{(11s+15)2^{3s+5}}.
\end{equation}   

Since our bound depend only on $\max\{A_1,\ldots,A_d,B_1,\ldots,B_e\}$, without 
loss of generality we now assume that 
$$
A_1=\ldots=A_d=B_1=\ldots =B_e = H. 
$$

We recall that an integer $k\ne 0$ is called $y$-smooth if all prime divisors of $k$ 
do not exceed $y$.


\begin{theorem}
\label{thm:main1} Let $H$, $d$, $e$ be  positive integers with $e\le d$. There exists a constant $b_d$
depending only on $d$, and an integer $Z$,  which is $O(H^{1/\gamma_{d+e+1}})$-smooth and satisfies  
$$
\log{Z}\ll H^{(d+e)(d+e+2)^2/4} \log{H} ,
$$
such that for each prime number $p\nmid Z$ the following holds.
For any generalised arithmetic progressions $\cA,\cB\subseteq \overline\F_p$ of the form 
\begin{align*}
\cA&=\left \{\alpha_0+ \alpha_1h_1+\ldots+\alpha_dh_d :~ 1\le h_i\le H, \ i =1, \ldots, d\right \},\\
\cB&=\left \{\beta_0+ \beta_1j_1+\ldots+\beta_ej_e :~ 1\le j_i\le H, \ i =1, \ldots, e\right \},
\end{align*}
and $\lambda \in \overline\F_p^{*}$ we have
$$
I_p(\cA,\cB,\lambda)\le \exp\left(b_d \log{H}/\log\log{H}\right). 
$$
\end{theorem}

The integer $Z$ in Theorem~\ref{thm:main1} is constructed explicitly in Lemma~\ref{lem:iter} below and is divisible by all primes $p \le  H^{d+e+o(1)}$ (note that $o(1)$ here denotes a negative quantity).  This is established at  the end of the proof of Lemma~\ref{lem:iter}.



From Theorem~\ref{thm:main1} we may deduce an estimate which holds for all primes provided our generalised arithmetic progressions are not too large. We also obtain better results for almost all primes. In particular, using the fact that no primes $p\ge Z$ divide a $Z$-smooth integer, we obtain:

\begin{cor}
\label{cor:all-primes}
Let notation be as in Theorem~\ref{thm:main1}. For any prime $p$ and integer $H$ satisfying
$$H\le C_0(d) p^{\gamma_{d+e+1}},$$
for some constant $C_0(d)$ depending only on $d$, 
we have 
$$
I_p(\cA,\cB,\lambda)\le \exp\left(b_d \log{H}/\log\log{H}\right). 
$$
\end{cor} 

As a second application, using the fact that any integer $Z$ has at most $O\(\log{Z}/\log\log{Z}\)$ prime divisors, we obtain:

\begin{cor}
\label{cor:almost-all-primes}
Let notation be as in Theorem~\ref{thm:main1}. For all but at most 
$O\(H^{(d+e)^3+(d+e)}\)$ 
primes $p$, we have 
$$
I_p(\cA,\cB,\lambda)\le \exp\left(b_d \log{H}/\log\log{H}\right).  
$$
\end{cor}

We note an important feature of Theorem~\ref{cor:almost-all-primes} is the set of 
primes  is independent of the generalized arithmetic progressions $\cA,\cB$. 

Corollaries~\ref{cor:all-primes} and~\ref{cor:almost-all-primes} immediately yield an estimate for equations with Kloosterman fractions and squares.  
Indeed using that over any field and $\lambda \ne 0$, if
$$
a^{-1}+ b^{-1} = \lambda
$$
then 
$$
(a -\lambda^{-1})(b-\lambda^{-1}) = \lambda^{-2}, 
$$
and over any algebraically closed field, if $\lambda\neq 0$ and $a,b$ satisfy
$$a^2+b^2=\lambda,$$
then 
$$(a +ia )(a -ib)=\lambda,$$
where $i$ is a square root of $-1$, we obtain the following results.  

\begin{cor}
\label{cor:main22}
With notation and conditions as in either Corollary~\ref{cor:all-primes} or Corollary~\ref{cor:almost-all-primes}, the number of solutions to
the equations
$$a^{-1}+ b^{-1} =  \lambda , \qquad a\in \cA, b\in \cB,
$$
and
$$a ^2+b^2 =  \lambda , \qquad a\in \cA, b\in \cB,
$$
in $\overline \F_p$
are bounded by   $H^{o(1)}$.  
\end{cor}

We remark that our method, with minor changes,  can allow us to extend 
our results to equations 
\begin{equation}
\label{eq:multi-problem}
a_1\ldots a_\nu = \lambda, \qquad a_i \in \cA_i, \ i =1, \ldots, \nu, \quad a_i\in \cA_i,
\end{equation}
 with any $\nu \ge 2$ and generalised arithmetic progressions $\cA_1, \ldots, \cA_\nu \subseteq \F_p$. A direct application of such techniques gives a poor dependence on the parameter $\nu$. An interesting problem is to determine the largest real numbers $\gamma_{\nu,d}$ such that the number of solutions to~\eqref{eq:multi-problem} is bounded by $(|\cA_1|\ldots|\cA_\nu|)^{o(1)}$ provided $\cA_1,\ldots,\cA_\nu$ are generalized arithmetic progressions of rank at most $d$ satisfying 
$$|\cA_i|\ll p^{\gamma_{\nu,d}}.$$

\subsection{Applications to  the Erd\H{o}s--Szemer{\'e}di conjecture in finite fields}
\label{sec:appl} 
As usual, given a set $\cA\subseteq \cG$ with a group operation $\ast$, we write 
$$
\cA \ast \cA = \left \{a\ast b:~ a,b \in \cA\right \}.
$$
Clearly for sets in rings we can use $\ast \in \left \{+, \times\right \}$. 

Here we also denote 
$$
\cA^{-1}  = \{a^{-1}:~a \in \cA\}, \quad \cA^2=\{ a^2:~a\in \cA\}.
$$

Combining  the above results 
with some modern results~\cite[Theorem~4]{CS} of additive combinatorics towards the celebrated theorem of Freiman~\cite{Frei}, 
we, in particular verify the Erd\H{o}s--Szemer{\'e}di conjecture for sets with small sumset and small cardinality.
This can be considered an extension of some ideas of Chang~\cite{Chang} into the setting of prime 
finite fields.

\begin{theorem}
\label{thm:main24}  For any fixed $K \ge 2$
and 
$$
\delta=\frac{1}{(44K+26)2^{12K+8}}, 
$$
there exist some constants $b_0(K)$ and $c_0(K)$, depending only on $K$,  such that for each prime $p$, if $\cA \subseteq \mathbb{F}_p$ satisfies 
$$|\cA+\cA|\le K|\cA| \mand |\cA|\le c_0(K)p^{\delta}$$
then for any $\lambda \in \F_p^{*}$  the number of solutions to each of the equations 
$$
  a_1a_2=  \lambda , \qquad  a_1^{-1}+ a_2^{-1} = \lambda, 
  \qquad  a_1^{2}+ a_2^{2} = \lambda
$$
with variables $a_1,a_2\in \cA$ is $\exp\left(b_0(K)\log{|\cA|}/\log\log{|\cA|}\right)$.
\end{theorem}  

An immediate consequence of Theorem~\ref{thm:main24} is an estimate for the cardinality of sets related to the  Erd\H{o}s--Szemer{\'e}di conjecture.
Indeed, using Theorem~\ref{thm:main24} one has that
$$
|\cA|^2=\sum_{\lambda \in \cA\cA}I_p(\cA,\cA,\lambda) \le 2 |\cA|+\sum_{\substack{\lambda\in \cA\cA \\ \lambda \not \equiv 0\mod{p}}}I_p(\cA,\cA,\lambda)\le |\cA\cA|  |\cA|^{o(1)}.
$$
A similar argument also works for the sets $\cA^{-1} + \cA^{-1}$ and we obtain the following result.

\begin{cor}
\label{cor:main24}
With notation and conditions as in Theorem~\ref{thm:main24}, for any fixed $K$ we have 
$$|\cA\cA|\ge |\cA|^{2+o(1)} \mand  |\cA^{-1} + \cA^{-1}| \ge |\cA|^{2+o(1)}.$$
\end{cor}  

We note that Corollary~\ref{cor:main24} is a step towards a positive answer to a question raised by  Murphy, Petridis, 
Roche-Newton, Rudnev and Shkredov~\cite[Question~2]{MPRRS} whether for any $\varepsilon > 0$ there exists some 
$\eta(\varepsilon)$ depending only on $\varepsilon$ with $\eta(\varepsilon) \to 0$ as $\varepsilon \to 0$,  such that
 if $\cA \subseteq \mathbb{F}_p$ satisfies 
$|\cA+\cA|\le |\cA|^{1+\varepsilon}$ then 
$$|\cA\cA| \ge |\cA|^{2- \eta(\varepsilon)}.
$$
 Theorem~\ref{thm:main24} confirms this in  the extreme case of rapidly decaying (as $|\cA|$ grows) values of  $\varepsilon$. In other words instead of fixed $K$ in Corollary~\ref{cor:main24} we can take $K$ as a very slowly growing function of $ |\cA|$.
 We also recall that  Shkredov~\cite[Corollary~2]{Shk} has shown that if 
\begin{equation}
\label{eq:ShkAcond}
|\cA+\cA|\ll |\cA|
\end{equation}
 for a set  $\cA \subseteq \mathbb{F}_p$ 
 of cardinality $|\cA| \ll p^{13/23}$  
 then the number of solutions to  
$$
a_1a_2 =  \lambda, \qquad a_1,a_2\in \cA,
$$
is bounded by $|\cA|^{149/156+o(1)}$.  
Clearly,  this result and  Theorem~\ref{thm:main24}  are of similar spirit, 
however they are incomparable. In particular, the cardinality of the sets considered in~\cite[Corollary~2]{Shk} is uniform with respect to the implied constant in~\eqref{eq:ShkAcond}, which is a feature not present in our bound. We refer the reader to~\cite{MPW} for various incidence results related to counting solutions to multiplicative equations with variables belonging to sets with small sumset. Our result does give a direct improvement on Grosu~\cite[Section~4]{Gros},   
who obtains similar estimates to Theorem~\ref{thm:main24} with the condition, which 
we slightly simplify as
\begin{equation}
\label{eq:GrosuCong}
|\cA|\le \frac{1}{\log 2} \log \log\log p-1-\varepsilon,
\end{equation}
for any   $\varepsilon > 0$ provided that $p$ is large enough. 
However, the paper of Grosu~\cite{Gros} contains other interesting results which allow one to lift problems in $\F_p$ to $\C$ while preserving  more arithmetic information than counting solutions to equations considered in Theorem~\ref{thm:main24}.

We now obtain a version of Theorem~\ref{thm:main24} which holds for almost all primes. 

\begin{theorem}
\label{thm:main24-AA}   Let  $ A \ge 3$ be sufficiently large  and let  $K \ge 2$  be a fixed integer.
For all but at most   $O\(A^{8K^3+4K^2} \log A/\log \log A\)$  primes $p$   
with 
$$
p  > c_0(K) A^{2K} 
$$
for some sufficiently large constant $c_0(K)$ depending only on $K$, 
the following holds. 
If $\cA \subseteq \mathbb{F}_p$ satisfies 
$$|\cA+\cA|\le K|\cA| \mand |\cA|\le A$$
then for any $\lambda \in \F_p^{*}$  the number of solutions to each of the equations 
$$
 a_1a_2=  \lambda   \mand a_1^{-1}+ a_2^{-1} = \lambda
$$
with variables $a_1,a_2\in \cA$ is $|\cA|^{o(1)}$.
\end{theorem}

As before, we obtain a result towards  the Erd\H{o}s--Szemer{\'e}di conjecture 
modulo almost all primes. 

\begin{cor}
\label{cor:main24-AA}
With notation and conditions as in Theorem~\ref{thm:main24-AA} we have 
$$|\cA\cA|  \ge |\cA|^{2+o(1)} \mand  |\cA^{-1} + \cA^{-1}| \ge |\cA|^{2+o(1)}.$$
\end{cor}  
 
\subsection{Overview of our approach}

We first illustrate the main ideas of our paper in the setting of Corollary~\ref{cor:all-primes}. With $H$ as in Theorem~\ref{thm:main1}, let $\cA,\cB\subseteq \F_{p}$ be generalised arithmetic progressions of rank $d,e$ respectively  
and recall we aim to show  
$$I_p(\cA,\cB)=H^{o(1)}.$$ 
Our main input is the following iterative inequality, see Lemma~\ref{lem:iter}, that there exists generalised arithmetic progressions $\widetilde \cA, \widetilde \cB$ of rank $\widetilde d, \widetilde e$ respectively, satisfying 
$$I_p(\cA,\cB)\ll H^{o(1)}I_p(\widetilde \cA,\widetilde \cB),$$
and 
$$\widetilde d+\widetilde e<d+e, \qquad |\widetilde \cA||\widetilde \cB|\ll |\cA|\cB|.$$
Proceeding by induction on $d+e$, the above properties are sufficient to establish the desired result.
 Suppose $\cA,\cB\subseteq \F_p$ are given by 
\begin{align*}
\cA&=\left \{\alpha_0+ \alpha_1h_1+\ldots+\alpha_dh_d :~ 1\le h_i\le H, \ i =1, \ldots, d\right \},\\
\cB&=\left \{\beta_0+ \beta_1j_1+\ldots+\beta_ej_e :~ 1\le j_i\le H, \ i =1, \ldots, e\right \},
\end{align*}
and for simplicity assume $\cA,\cB$ are proper. Hence we aim to count the number of solutions to the equation 
\begin{equation}
\begin{split} 
\label{eq:eqn-o}
(\alpha_0+\alpha_1h_{1}+\dots+\alpha_d h_{d})&(\alpha_0+\beta_1j_{1}+\dots+\beta_d j_{e})= \lambda,  \\
 1\le h_1,\ldots,h_d&,  j_1, \ldots, j_e \le H.
\end{split} 
\end{equation}
Fix a pair 
$$(h_{1,0},\dots,h_{d,0}) \in [1,H]^{d}, \quad (j_{1,0}\dots,j_{e,0}) \in [1,H]^{e},$$
satisfying 
$$(\alpha_0+\alpha_1h_{1,0}+\dots+\alpha_d h_{d,0})(\beta_0+\beta_1j_{1,0}+\dots+\beta_d j_{e,0})= \lambda,$$
and consider the variety $V_p\subseteq \overline \F^{d+e+2}$ defined by the system of equations 
\begin{align*}
& (X_0+X_1h_{1}+\dots+X_d h_{d})(Y_0+Y_1j_{1}+\dots+Y_e j_{e})- \\ & \quad \quad \quad 
(X_0+X_1h_{1,0}+\dots+X_d h_{d,0})(Y_0+Y_1j_{1,0}+\dots+Y_e j_{e,0})=0, 
\end{align*}
such that $h_1,\dots,h_d,j_1,\dots,d_e$ satisfy~\eqref{eq:eqn-o} and in variables $X_0,\dots,Y_e$. Let $V$ denote the corresponding variety over $\C$. By assumption, we have 
$$(\alpha_0, \dots, \alpha_d, \beta_0,\dots,\beta_e)\in V_p.$$
Assuming $H$ is sufficiently small in terms of $p$,  a local-global result of D'Andrea, Ostafe, Shparlinski and Sombra~\cite{DOSS}, see Lemma~\ref{lem:TAMS} below, implies there exists 
$$(\rho_0,\dots,\rho_d,\gamma_0,\dots,\gamma_e)\in V.$$
Hence any solution $h_1,\dots,j_e$ to~\eqref{eq:eqn-o} also satisfies 
\begin{align*}
(\rho_0+\rho_1h_{1}+\dots+\rho_d h_{d})(\gamma_0+\gamma_1j_{1}+\dots+\gamma_e j_{e})= \lambda_0,
\end{align*}
for some $\lambda_0\in \C$. A result of Chang~\cite{Chang}, see Lemma~\ref{lem:chang} below, implies there exists $H^{o(1)}$ possible values for either 
\begin{equation}
\label{eq:rho-0}
\rho_0+\rho_1h_{1}+\dots+\rho_d h_{d}=\mu_1,
\end{equation}
or 
$$\gamma_0+\gamma_1j_{1}+\dots+\gamma_e j_{e}=\mu_2.$$
Assuming~\eqref{eq:rho-0}, our set of solutions to~\eqref{eq:eqn-o} is restricted to the union of $H^{o(1)}$ cosets of a lattice $\cL$ of rank smaller than $d$. After performing basis reduction to $\cL$ and back-substitution, the desired iterative inequality follows.  
\section{Preliminaries}

\subsection{Tools from Diophantine geometry}

For a polynomial $G$ with integer coefficients, its \textit{height}, is defined as the logarithm of the maximum of 
the absolute values of the coefficients of $G$. The height of an algebraic number $\alpha$  is defined as the height of its minimal polynomial (we also set it to $1$ for $\alpha=0$).

We now recall the statement of~\cite[Theorem~2.1]{DOSS} which underlies our approach. 

\begin{lemma}
\label{lem:TAMS}
 Let $G_i \in \mathbb{Z}[T_1,\ldots,T_n]$, $i=1,\ldots,s$, $n \geq 1$ be polynomials of degree at most $r \geq 2$ and height at most $h$, whose zero set in $\mathbb{C}^n$ has a finite number $\kappa$ of distinct points.
 Then there is an integer $\mathfrak{A} \geq 1$ with
$$
\log \mathfrak{A} \leq (11n +4)r^{3n+1}h + (55r + 99)\log ((2n+5)s)r^{3n+2}
$$ 
such that, if $p$ is a prime not dividing $\mathfrak{A}$, then the zero set in $\overline{\mathbb{F}}^n_p$ of 
 the polynomials $G_i$ reduced modulo $p$, $ i=1,\ldots,s$, consists of exactly $\kappa$ distinct points.
\end{lemma}

Results of this type have previously appeared. For
 example Chang~\cite[Lemma~2.14]{Chang} has shown
the following result. 
Let
$$
\cV=\bigcap_{j=1,\ldots,s}[F_j=0],
$$
 be an affine variety in $\C^{n}$ defined by polynomials $F_j \in \Z[X_1,\ldots,X_n]$, $j=1, \ldots, s$, of height at most $h$
 and let $F\in \Z[X_1,\ldots,X_n]$ be a polynomial
of height at most $h$ such that there is $\balpha \in \cV$ with $F(\balpha)\neq 0$.
Then there is $\bbeta \in \cV$  with 
$F(\bbeta)\neq 0$ whose coordinates are algebraic numbers of height $O(h)$.
%

There are also modulo $p$ analogues of~\cite[Lemma~2.14]{Chang} 
which allow one to lift solutions to $\C$ from a variety modulo $p$ and we refer the reader to~\cite{Gros} for results of this type. One may also use effective versions of the B\'{e}zout identity, and more generally the Hilbert Nullstellensatz, to lift points on a variety modulo $p$ to $\C$, and this idea has previously been used 
in~\cite{BBK,BGKS,CKSZ, KMSV, Shp2}. 


\subsection{Tools from geometry of numbers}

Let $\left \{\vec{b}_1,\ldots,\vec{b}_m\right \}$ be a set of $m\le d$ linearly independent vectors in
${\R}^d$. The set
of vectors
$$ \cL = \left\{ \sum_{i=1}^m n_i \vec{b}_i :~ n_i \in \Z \right \},$$
is called an $d$-dimensional lattice of rank $m$. The set
$\left \{ \vec{b}_1,\ldots, \vec{b}_m\right \}$ is called a \textit{basis} of $\cL$. Each lattice has multiple sets of basis vectors, and we refer to any other set 
$\{\widetilde{\vec{b}}_1,\ldots,\widetilde{\vec{b}}_m\}$ of linearly independent vectors such that 
$$ \cL = \left\{ \sum_{i=1}^m n_i \widetilde{\vec{b}}_i :~ n_i \in \Z \right \}$$
as a basis. 
 We also define the determinant of $\cL$ as 
$$
\det \cL = \sqrt{\left|\det B\cdot B^T\right|},
$$
where $B$ is the $(m\times d)$-matrix with rows $\vec{b}_1,\ldots,\vec{b}_m$, and is independent of the choice of basis. We refer to~\cite{GrLoSch} for a background on lattices.

The following is~\cite[Lemma~1]{HB}.

\begin{lemma}
\label{lem:HB}
Let $\cL\subseteq \Z^d$ be a lattice of rank $m$.  Then $\cL$ has a basis $\vec{b}_1,\ldots,\vec{b}_m$ such that, for each $\vec{x}\in \cL$, we may write 
$$\vec{x}=\sum_{j=1}^{m}\lambda_j\vec{b}_j,$$
with 
$$\lambda_j\ll \frac{\|\vec{x}\|}{\|\vec{b}_j\|}.$$
We also have 
$$\det \cL\ll \prod_{j=1}^{m}\|\vec{b}_i\|\ll \det \cL.$$ 
\end{lemma}

\begin{lemma}
\label{lem:linear}
Let $\alpha_1,\ldots,\alpha_d\in \C$ and  let $\cL$ denote the lattice 
$$
\cL=\{ (n_1,\ldots,n_d)\in \Z^d :~ \alpha_1n_1+\ldots+\alpha_dn_d=0\}.
$$
For integers $H_1,\ldots,H_d$ we  consider  the convex body 
$$
D=\{ (x_1,\ldots,x_d)\in \R^{d} :~  |x_i|\le H_i \}.
$$ 
If $\cL\cap D$ contains $d-1$ linearly independent points and there exists some $1\le \ell\le d$ such that $\alpha_\ell\neq 0$, then there exists some $1\le j \le d$ such that  for each $i = 1, \ldots, d$ there exist integers $a_i$ and $b_i$ satisfying 
$$
\frac{\alpha_i}{\alpha_j}=\frac{a_i}{b_i}, \qquad \gcd(a_i,b_i)=1, \qquad a_i, b_i \ll H^{d}, 
$$
where 
$$
H=\max_{1\le i \le d} H_i.
$$
\end{lemma}

\begin{proof}
Choose $d-1$  linearly independent points $\vx^{(1)},\ldots, \vx^{(d-1)}$ satisfying
$$\vx^{(i)}=(x_{i,1},\ldots,x_{i,d})\in \cL\cap D, \qquad 1\le i \le d-1.$$
Let $X$ denote the $(d-1)\times d$ matrix whose $i$-th row is $\vx^{(i)}$ and let $X^{(j)}$ denote the $(d-1)\times (d-1)$ matrix obtained from $X$ by removing the $j$-th column. By assumption, the rank of $X$ equals $d-1$. Hence there exists some $1\le j \le d$ such that 
\begin{equation}
\label{eq:det}
\det X^{(j)}\neq 0.
\end{equation}
By symmetry we may suppose $j=d$. Since each $\vx^{(i)}\in \cL\cap D$, we have 
$$
X^{(d)}\begin{pmatrix} \alpha_1 \\ \alpha_2 \\ \vdots \\ \alpha_{d-1} \end{pmatrix}=-\alpha_d \begin{pmatrix} x_{1,d} \\ x_{2,d} \\ \vdots \\ x_{d-1,d} \end{pmatrix}.
$$
Note that~\eqref{eq:det} and the assumption $\alpha_{\ell}\neq 0$ implies $\alpha_{d}\neq 0$. Let $Y^{(d)}$ denote the adjoint matrix of $X^{(d)}$, 
thus 
$$
  X^{(d)} Y^{(d)}=\det X^{(d)} I_{d-1},
 $$
where $I_{d-1}$ is the $(d-1)\times (d-1)$-identity matrix.
Hence,  the above implies 
$$
\det X^{(d)}\begin{pmatrix} \alpha_1 \\ \alpha_2 \\ \vdots \\ \alpha_{d-1} \end{pmatrix}=-\alpha_d Y^{(d)}\begin{pmatrix} x_{1,d} \\ x_{2,d} \\ \vdots \\ x_{d-1,d} \end{pmatrix}.
$$
By  Hadamard's inequality  and the definition of $H$, 
$$
\det X^{(d)}\ll H^{d},
$$
and 
$$
Y^{(d)}\begin{pmatrix} x_{1,d} \\ x_{2,d} \\ \vdots \\ x_{d-1,d} \end{pmatrix}=\begin{pmatrix} y_{1} \\ y_{2} \\ \vdots \\ y_{d-1} \end{pmatrix}, 
$$
for some integers $y_1,\ldots,y_{d-1} \ll H^{d},$ from which the result follows.
\end{proof}

\subsection{Tools from additive combinatorics}
\label{sec:additivecomb}

Our proof of Theorem~\ref{thm:main1} uses Lemma~\ref{lem:TAMS} to reduce to counting solutions to multiplicative equations over $\C$ to which the following result of Chang~\cite[Proposition~2]{Chang} may be applied, see also~\cite[Remark~1]{Chang}.

\begin{lemma}  
\label{lem:chang}
For each integer $d\ge 1$ there exist a constant $B_d$, depending only on $d$, such that the following holds. Let $\gamma_0,\ldots,\gamma_d\in \C$ and define the set $\cA$ by 
$$
\cA=\{ \gamma_0+\gamma_1h_1+\ldots+\gamma_dh_d :~ |h_i|\le H_i\}.
$$
For any $\lambda \in \C^{*}$ the number of solutions to 
$$
a_1a_2=\lambda, \quad a_1,a_2\in \cA,
$$
is bounded by $\exp\left(B_d\log{|\cA|}/\log\log{|\cA|}\right)$.
\end{lemma}
 
\section{An iterative inequality}

\subsection{Formulation of the result} 
Our main input for the proof of Theorem~\ref{thm:main1} is the following iterative inequality which combines some ideas of Chang~\cite{Chang} with lattice basis reduction. Note that, as in~\cite{Chang}, it is not necessary to assume our generalised arithmetic progression is proper. 

Recall that for $\cA,\cB\subseteq \F_p$ and 
$\lambda\in \F_p$ we define $I_p(\cA,\cB,\lambda)$ by~\eqref{eq:IpA}.  

We also recall that an integer $n$ is called $y$-smooth if all prime divisors $p$ of $n$ satisfy $p \le y$.


\begin{lemma}
\label{lem:iter}
Let $H$, $d$, $e$ be  positive integers with $e\le d$  and let $H$ be sufficiently large.  
There exists a constant  $B_d$   depending only on $d$, and an integer $Z_{d,e}$ which
\begin{itemize}
\item[(i)] 
 is $O(H^{1/\gamma_{d+e+1}})$-smooth with $\gamma_{d+e+1}$ given by~\eqref{eq:gamma}, 
\item[(ii)] is divisible by all  primes $p \le   H^{d+e+o(1)}$,  
\item[(iii)]   satisfies 
$$
\log{Z_{d,e}}\ll H^{(d+e)(d+1)(e+1)} \log{H},
$$ 
\end{itemize}
such that for any prime $p \nmid  Z_{d,e}$ the following holds. 
Let $\lambda\in \overline\F_p^*$
and $\cA,\cB\subseteq \overline\F_p$  generalised arithmetic progressions of the form \begin{equation}
\label{eq:Aiter}
\cA=\{ \alpha_0+\alpha_1h_1+\ldots +\alpha_d h_d :~ |h_i|\le H, \ i =1, \ldots, d\},
\end{equation}
and 
\begin{equation}
\label{eq:Biter}
\cB=\{ \beta_0+\beta_1j_1+\ldots +\beta_e j_e :~ |j_i|\le H, \ i =1, \ldots, e\},
\end{equation}
with $d,e\ge 2$ and 
$$\alpha_1,\ldots,\alpha_d,\beta_1,\ldots,\beta_e \in \overline\F_p^*.
$$ 
 There exists a constant $\tCde$ depending only on $d$ and $e$,  integers $\td$ and $\te$ satisfying 
$$
\td \le d, \qquad \te \le e, \qquad  \td+\te<d+e,
$$ 
 generalised arithmetic progressions $\tcA,\tcB$ of the form
\begin{align*}
& \tcA=\{ \talpha_0+\talpha_1h_1+\ldots +\talpha_{\td} h_{\td} :~ |h_i|\le \tCde H, \  i =1, \ldots, \td\},\\
& \tcB=\{ \tbeta_0+ \tbeta_1j_1+\ldots + \tbeta_{\te} j_{\te} :~ |j_i|\le \tCde H, \ i =1, \ldots, \te \},
\end{align*}
with 
$$\talpha_1 \ldots,\talpha_d, \tbeta_1 \ldots \tbeta_e \in \overline\F_p^{*},$$
and some $\mu \in \overline\F_p^*$ such that 
$$
I_p(\cA,\cB,\lambda)\le \exp\left(B_d \log{H}/\log\log{H}\right)I_p(\tcA,\tcB,\mu ).
$$
\end{lemma}

We split the proof of Lemma~\ref{lem:iter} in a series of steps. 

\subsection{Elimination undesired primes} 
We first denote 
\begin{equation}
\label{eq:Z0def}
Z_0 =\prod_{p\le C_d H^{d}}p
\end{equation}
for an appropriately large constant $C_d$, which depends only on $d$. 
We now fix $p \nmid Z_0$, thus 
 \begin{equation}
\label{eq:large p}
p >   C_d H^{d}.
\end{equation}  

We first construct the integer $Z_{d,e}$. For $\vh,\vh_0 \in \Z^{d}$ and $\vj,\vj_0 \in \Z^{e}$ define the polynomial
\begin{equation}
\begin{split} 
\label{eq:Phhjj}
&P_{\vh,\vh_0, \vj,\vj_0}\(\vec{X},\vec{Y} \)\\
&\qquad =(X_0+X_1h_{0,1}+\ldots+X_d h_{0,d})(Y_0+Y_1 j_{0,1}+\ldots+Y_e j_{0,e}) \\ 
& \qquad  \qquad     -(X_0+X_1h_1+\ldots+X_d h_d)(Y_0+Y_1 j_{1}+\ldots+Y_e j_{e}).
\end{split}
\end{equation}
We may identify the polynomial  $P_{\vh,\vh_0 \vj,\vj_0}(\vec{X},\vec{Y} )$ with a point in the vector space $\C^{\Delta}$
where 
$$
\Delta = (d+1)(e+1)-1 =  de+d +e,
$$  
which is formed by its coefficients. Suppose $\cM\subseteq \Z^{d}\times \Z^{e}$  satisfies  $|\cM|\le \Delta $ and the set
$$\{P_{\vh,\vh_0, \vj,\vj_0}(\vec{X},\vec{Y} ) :~ (\vh,\vj)\in \cM\},$$
is linearly independent over $\C$. 

Let  $M\(\vh_0,\vj_0,\cK\)$ denote the $|\cM|\times \Delta$ matrix whose rows correspond to coefficients of the polynomials $P_{\vh,\vh_0, \vj,\vj_0}(\vec{X},\vec{Y} )$ with $(\vh,\vj)\in \cM.$ Define 
\begin{equation}
\label{eq:Z1def}
Z_1\(\vh_0,\vj_0,\cM\)=\det M_0\(\vh_0,\vj_0,\cM\),
\end{equation}
where  $M_0\(\vh_0,\vj_0,\cM\)$ is a $|\cM|\times |\cM|$ submatrix of $M\(\vh_0,\vj_0,\cM\)$ with nonzero determinant. If 
$$\vh_0\in [-H,H]^{d}, \qquad \vj_0\in [-H,H]^{e}, \qquad \cM\subseteq [-H,H]^{d}\times [-H,H]^{e}
$$ then for each $(\vh,\vj)\in \cM$ the polynomial $P_{\vh,\vh_0, \vj,\vj_0}$ has height at most $2\log{H}+O(1)$. 
Clearly there are 
\begin{equation}
\label{eq:Choices}
\begin{split}
W = (2H+1)^d \cdot (2H+1)^e \cdot & \sum_{r=1}^{\Delta }  \binom{(2H+1)^{d+e}}{r}\\
& \ll H^{d+e +  \Delta (d+e)} = H^{(d+e)(\Delta +1)} 
\end{split} 
\end{equation}
choices for the above triple $\(\vh_0,\vj_0,\cM\)$. 

By  Hadamard's inequality
\begin{equation}
\label{eq:Z1b}
Z_1\(\vh_0,\vj_0,\cK\)\ll H^{2|\cM|}\ll H^{2\Delta }.
\end{equation}
 Define 
\begin{equation}
\label{eq:Z1}
Z_1=\prod_{\substack{\vh_0,\in [-H,H]^{d} \\ \vj_0\in [-H,H]^{e} \\ \cM\subseteq [-H,H]^{d}\times [-H,H]^{e} \\ |\cM|\le \Delta }}Z_1\(\vh_0,\vj_0,\cM\),
\end{equation}  
so, recalling~\eqref{eq:Choices} and~\eqref{eq:Z1b}, we see that  
$$
\log{Z_1}\ll W \log H \ll H^{(d+e)(\Delta +1)} \log{H},
$$
and that  $Z_1$ is $O(H^{2\Delta })$-smooth. 

For each $\vh_0,\vj_0,\cM$ as above, let $V\(\vh_0,\vj_0,\cM\)$ denote the variety
\begin{align*}
V\(\vh_0,\vj_0,\cM\)& =\bigcap_{(\vh, \vj)\in \cM}\left\{  (\vx,\vy)\in  \C^{d+1}\times \C^{e+1} :~  P_{\vh,\vh_0, \vj,\vj_0}(\vx,\vy)= 0 \right \} \\
& \qquad \qquad \quad   \bigcap \left\{  (\vx,\vy)\in \C^{d+1}\times \C^{e+1} :~ x_1=1,~ y_1=1 \right \},
\end{align*}
and let $V_p\(\vh_0,\vj_0,\cM\)$ be the reduction of $V\(\vh_0,\vj_0,\cM\)$ modulo $p$
\begin{align*}
V_p\(\vh_0,\vj_0,\cM\)& =\bigcap_{(\vh, \vj)\in \cM}\left\{  (\vx,\vy)\in  \overline \F_p^{d+1}\times \overline \F_p^{e+1} :~  P_{\vh,\vh_0, \vj,\vj_0}(\vx,\vy)= 0 \right \} \\
& \qquad \qquad \quad   \bigcap \left\{  (\vx,\vy)\in \overline \F^{d+1}\times \overline \F_p^{e+1} :~ x_1=1,~ y_1=1 \right \}.
\end{align*}
If $|V\(\vh_0,\vj_0,\cM\)|=\infty$ then define 
$$Z_2\(\vh_0,\vj_0,\cM\)=1.$$

Otherwise, that is, if $|V\(\vh_0,\vj_0,\cM\)|<\infty$,  by Lemma~\ref{lem:TAMS} there exists a positive integer  $Z_2\(\vh_0,\vj_0,\cM\)$ satisfying 
$$
Z_2\(\vh_0,\vj_0,\cM\)\ll H^{1/\gamma_{d+e+1}},
$$
such that for each prime $p$ not dividing $Z_2\(\vh_0,\vj_0,\cM\)$, we have 
$$|V\(\vh_0,\vj_0,\cM\)|=|V_p\(\vh_0,\vj_0,\cM\)|.$$
Denote
\begin{equation}
\label{eq:Z2}
Z_2=\prod_{\substack{\vh_0,\in [-H,H]^{d} \\ \vj_0\in [-H,H]^{e} \\ \cM\subseteq [-H,H]^{d}\times [-H,H]^{e} \\ |\cM|\le \Delta }}Z_2\(\vh_0,\vj_0,\cM\). 
\end{equation}

With $Z_1$ and $Z_2$ as in~\eqref{eq:Z1} and~\eqref{eq:Z2} we define
$$
Z_{d,e}=Z_0Z_1 Z_2.
$$
Since $\gamma_{d+e+1} \le \Delta^{-1}$ we see that $Z_{d,e}$ is $O(H^{1/\gamma_{d+e+1}})$-smooth and satisfies 
$$
\log{Z_{d,e}}\ll H^{(d+e)(\Delta +1 )} \log{H}.
$$ 

From now on we only consider primes $p \nmid Z_{d,e}$.

\subsection{Local to global lifting  of rational points on some varieties} 
Fix some prime $p$ not dividing $Z_{d,e}$. 
With $\cA,\cB$ as in~\eqref{eq:Aiter} and~\eqref{eq:Biter}, choose 
\begin{align*}
& \cH \subseteq [-H,H]^{d}\cap \Z^{d}, \\
& \cJ\subseteq [-H,H]^{e}\cap \Z^{e},
\end{align*}
 such that the points 
\begin{align*}
& \alpha_0+h_1\alpha_1+\ldots+h_d\alpha_d, \quad (h_1,\ldots,h_d)\in \cH,\\
& \beta_0+j_1\beta_1+\ldots+j_e\beta_e, \quad (j_1,\ldots,j_e)\in \cJ,
\end{align*}
are distinct modulo $p$ and for each $a\in \cA$ there exists some integer vector $(h_1,\ldots,h_d)\in \cH$ such that 
$$a=\alpha_0+h_1\alpha_1+\ldots+h_d\alpha_d$$ 
and for each $b\in \cB$ there exists some $(j_1,\ldots,j_e)\in \cJ$ such that 
$$b=\beta_0+j_1\beta_1+\ldots+j_e\beta_e.$$
 Write 
$$ 
\vh =(h_1,\ldots,h_d) \mand \vj=(j_1,\ldots,j_e),
$$
so that $I_p(\cA,\cB,\lambda)$ is bounded by the number of solutions to
\begin{equation}
\label{eq:eqn}
(\alpha_0+\alpha_1h_1+\ldots+\alpha_d h_d)(\beta_0+\beta_1 j_{1}+\ldots+\beta_e j_{e})\equiv \lambda \mod{p},
\end{equation}
with $\vh\in \cH$ and $\vj\in \cJ$. Dividing both sides of~\eqref{eq:eqn} by $\alpha_1\beta_1$ and modifying $\alpha_0,\ldots,\alpha_d,\beta_0,\ldots,\beta_e,\lambda$ if necessary, we may assume 
\begin{equation}
\label{eq:alpha1}
\alpha_1=\beta_1=1.
\end{equation}
This reduction allows for a convenient application of Lemma~\ref{lem:TAMS}. In what follows, we will construct a variety over $\overline \F_p$ which contains the point  $(\alpha_0,\dots,\alpha_d,\beta_0,\dots,\beta_d)$ and the assumption that $\alpha_1=\beta_1=1$ allows us to obtain a nonzero point in the corresponding variety over $\C$ after applying Lemma~\ref{lem:TAMS}.

Let $\cK\subseteq \cH\times \cJ$ denote the set 
$$
\cK=\{ \(\vh, \vj\)\in \cH\times \cJ :~ \vh, \vj  \ \text{satisfy~\eqref{eq:eqn}}\},
$$
so that 
\begin{equation}
\label{eq:IABK}
I_p(\cA,\cB,\lambda)=|\cK|.
\end{equation}

Since we may assume $\cK\neq \emptyset$, fix some $(\vh_0, \vj_0)\in \cK$ and for each $(\vh, \vj)\in \cK$ consider the polynomial $P_{\vh,\vh_0, \vj,\vj_0}\(\vec{X},\vec{Y} \)$, given by~\eqref{eq:Phhjj}.
Clearly for $(\vh_0, \vj_0) \ne (\vh, \vj)$  the polynomial $P_{\vh,\vh_0, \vj,\vj_0}\(\vec{X},\vec{Y} \)$
is not identical to zero over $\C$. Indeed, it is enough to consider the specialisations $P_{\vh,\vh_0, \vj,\vj_0}\(\(1, 0\ldots, 0\),\vec{Y} \)$  
and $P_{\vh,\vh_0, \vj,\vj_0}(\vec{X},\(1, 0\ldots, 0\))$ to see this. 
Furthermore,  if $p > 2H$ (which 
is guaranteed by our assumption)  then  $(\vh_0, \vj_0) \ne (\vh, \vj)$  implies  $(\vh_0, \vj_0) \not \equiv  (\vh, \vj) \mod p$ 
and we see that $P_{\vh,\vh_0, \vj,\vj_0}(\vec{X},\vec{Y} )$ 
is also not identical to zero over   $\overline \F_p$.

Let $V_p\subseteq \overline \F^{d+e+2}_p$ denote the variety \begin{align*}
V_p&=\bigcap_{(\vh,\vj)\in \cK}\left\{  (\vx,\vy)\in \overline \F^{d+1}_p\times \overline\F^{e+1}_p :~  P_{\vh,\vh_0, \vj,\vj_0}(\vx,\vy)= 0 \right \} \\ & \qquad \qquad \qquad  \quad \bigcap \left\{  (\vx,\vy)\in \overline \F^{d+1}_p\times \overline\F^{e+1}_p :~ x_1=1, ~ y_1=1 \right \}. 
\end{align*}
 Let $\cM \subseteq \cK$ be a maximal set  of $(\vh, \vj)\in \cK$, such that 
the  polynomials $P_{\vh,\vh_0, \vj,\vj_0}(\vec{X}, \vec{Y} )$
are linearly independent over $\C$. With $Z_1\(\vh_0,\vj_0,\cM\)$ defined as in~\eqref{eq:Z1def}, with respect to such $\cK$, since 
$$p \nmid Z_1\(\vh_0,\vj_0,\cM\)
$$
we conclude that  $\cM \subseteq \cK$ is also a maximal set  such that 
the  polynomials $P_{\vh,\vh_0, \vj,\vj_0}(\vec{X}, \vec{Y} )$
are linearly independent over $\overline \F_p$.  
Hence
\begin{align*}
V_p&=\bigcap_{(\vh,\vj)\in \cM}\left\{  (\vx,\vy)\in \overline \F^{d+1}_p\times \overline\F^{e+1}_p :~  P_{\vh,\vh_0, \vj,\vj_0}(\vx,\vy)= 0 \right \} \\ & \qquad \qquad \qquad  \quad \bigcap \left\{  (\vx,\vy)\in \overline \F^{d+1}_p\times \overline\F^{e+1}_p :~ x_1=1, ~ y_1=1 \right \} 
\end{align*}
 and $1 \le |\cM|\le \Delta$.  
  By definition of $\cK$ and~\eqref{eq:alpha1}, we have 
\begin{equation}
\label{eq:alphain}
(\alpha_0,\ldots,\alpha_d,\beta_0,\ldots,\beta_e)\in V_p.
\end{equation}
 Let $V\subseteq \C^{d+e+2}$ denote the variety
\begin{equation}
\begin{split} 
\label{eq:Vdef}
V& =\bigcap_{(\vh, \vj)\in \cM}\left\{  (\vx,\vy)\in  \C^{d+1}\times \C^{e+1} :~  P_{\vh,\vh_0, \vj,\vj_0}(\vx,\vy)= 0 \right \} \\
& \qquad \qquad \quad   \bigcap \left\{  (\vx,\vy)\in \C^{d+1}\times \C^{e+1} :~ x_1=1,~ y_1=1 \right \}.
\end{split} 
\end{equation}
We next show there exists some $(\brho,\btau)=(\rho_0,\rho_1,\ldots,\rho_d,\tau_0,\tau_1,\ldots,\tau_e)\in \C^{d+1}\times \C^{e+1}$ satisfying 
\begin{equation}
\label{eq:rho}
(\brho,\btau)\in V, \quad (\rho_1,\ldots,\rho_d)\neq \mathbf{0}, \quad (\tau_1,\ldots,\tau_e) \neq \mathbf{0}.
\end{equation}
 Certainly it is enough to show that
\begin{equation}
\label{eq:V nonempty}
|V|\ge 1,
\end{equation}
  as the non-vanishing conditions in~\eqref{eq:rho}
 are obvious  because any point $(\brho,\btau)\in V$ satisfies $\rho_1=\tau_1=1$.  
  
 We may assume 
\begin{equation}
\label{eq:Vbounded}
|V|<\infty,
\end{equation}
 since otherwise~\eqref{eq:rho} is trivial. 
We next apply  Lemma~\ref{lem:TAMS}.  The assumption~\eqref{eq:Vbounded} and that  
$$p \nmid  Z_2\(\vh_0,\vj_0,\cM\)
$$ 
implies
\begin{equation}
\label{eq:VVp}
|V|=|V_p|.
\end{equation}
We see from~\eqref{eq:alphain} that
$$
|V_p|\ge 1.
$$
Combining the above with~\eqref{eq:VVp}, we obtain~\eqref{eq:V nonempty}. 

 Hence there exists some $(\brho,\btau) \in \C^{d+1}\times \C^{e+1}$ satisfying~\eqref{eq:rho}. 
 Note that from~\eqref{eq:Vdef} we have 
\begin{equation}
\label{eq:rho1}
\rho_1=\tau_1=1.
\end{equation}

\subsection{Reduction to counting solutions to a multiplicative congruence on a complex line} 
We see that any solution to~\eqref{eq:eqn} satisfies 
\begin{equation}
\label{eq:rhorhorho}
(\rho_0+\rho_1h_1+\ldots+\rho_d h_d)(\tau_0+\tau_1 j_{1}+\ldots+\tau_e j_{e})=\vartheta,
\end{equation}
where 
$$
\vartheta=(\rho_0+\rho_1h_{0,1}+\ldots+\rho_d  h_{0,d})(\tau_0+\tau_1 j_{0,1}+\ldots+\tau_e j_{0,e}).
$$

Consider the following  two cases 
\begin{itemize}
\item   If $\vartheta=0$,   then either 
$$
\rho_0+\rho_1h_1+\ldots+\rho_d h_d=0,
$$
or 
$$
\tau_0+\tau_1j_1+\ldots+\tau_ e j_e=0.
$$

\item  If $\vartheta\neq 0$, then by Lemma~\ref{lem:chang}, there exists a set of 
$$\exp\left(B_d\log{H}/\log\log{H}\right)$$ pairs $\Omega=\{(\omega_1, \omega_2)\}$ such that any solution to~\eqref{eq:rhorhorho} satisfies 
$$
\rho_0+\rho_1h_1+\ldots+\rho_d h_d=\omega_1, \quad \tau_0+\tau_1j_1+\ldots+\tau_e j_e=\omega_2,
$$
for some $(\omega_1, \omega_2)\in \Omega$. 
\end{itemize}

Taking a maximum over the above two cases, we see that there exists some $\xi \in \C$ and some $i=1,2$ such that 
$$
I_p(\cA,\cB,\lambda)\le  \exp\left(B_d\log{H}/\log\log{H}\right)J_i(\cA,\cB,\lambda),
$$ 
where $J_1(\cA,\cB,\lambda)$ counts the number of solutions to 
\begin{equation}
\label{eq:eqn1}
(\alpha_0+\alpha_1h_1+\ldots+\alpha_d h_d)(\beta_0+\beta_1j_1+\ldots+\beta_e j_e)\equiv \lambda \mod{p},
\end{equation}
and 
\begin{equation}
\label{eq:zz1}
\rho_1h_1+\ldots+\rho_dh_d=\xi,
\end{equation} 
with variables $\vh\in \cH,\vj\in \cJ$ and $J_2(\cA,\cB,\lambda)$ counts the number of solutions 
to~\eqref{eq:eqn1}
and 
$$
\tau_1j_1+\ldots+\tau_ej_e=\xi,
$$
with variables $\vh\in \cH$, $\vj\in \cJ$.

Suppose first that 
\begin{equation}
\label{eq:IpJp}
I_p(\cA,\cB,\lambda)\le \exp\left(B_d\log{H}/\log\log{H}\right)J_1(\cA,\cB,\lambda),
\end{equation} 
the case 
\begin{equation}
\label{eq:IpJp1}
I_p(\cA,\cB,\lambda)\le \exp\left(B_d\log{H}/\log\log{H}\right)J_2(\cA,\cB,\lambda),
\end{equation} 
may be treated with a similar argument which we indicate at the end of the proof.

\subsection{Application of geometry of numbers to derive the desired inequality} 
Let $\cL$ denote the lattice 
$$
\cL=\left\{ (n_1,\ldots,n_d)\in \Z^{d} :~  \rho_1 n_1+\ldots+\rho_d n_d=0 \right\},
$$
and $D$ the convex body 
$$
D=\{ (n_1,\ldots,n_d) :~ |n_i|\le H\}.
$$ 
Assuming  $J_1(\cA,\cB,\lambda)\neq 0$, there exists some $\vh^{*}=(h^{*}_1,\ldots,h^{*}_d) \in D\cap \Z^d$ such that if $\vh=(h_1,\ldots,h_d)$ satisfies~\eqref{eq:zz1} then 
\begin{equation}
\label{eq:h123}
\vh-\vh^{*}\in \cL\cap 2D.
\end{equation}
By~\eqref{eq:rho1} we have 
$$
\dim \cL<d.
$$
Hence we may consider two cases, either 
\begin{equation}
\label{eq:case11}
\dim (\cL\cap 2D)<d-1,
\end{equation}
or 
\begin{equation}
\label{eq:case12}
\dim (\cL\cap 2D)=d-1.
\end{equation}

Suppose that we have~\eqref{eq:case11}. Let $\cL^{*}$ denote the lattice generated by $\cL\cap 2D$, so that $\dim \cL^{*}=r$ for some $r<d-1$.  By Lemma~\ref{lem:HB} there exists a basis $\blambda_1,\ldots,\blambda_{r},$ for $\cL^{*}$ such that each  $\vh$ satisfying~\eqref{eq:h123} may be expressed in the form 
\begin{equation}
\label{eq:glattice1}
\vh-\vh^{*}=k_1\blambda_1+\ldots+k_r \blambda_r,
\end{equation}
where from~\eqref{eq:h123} 
$$
k_1, \ldots, k_r \ll \frac{\|\vh-\vh^{*}\|}{\|\blambda_j\|} \ll H.
$$

Substituting~\eqref{eq:glattice1} into~\eqref{eq:eqn1}, there exists $\talpha_0,\ldots,\talpha_r\in \F_p$ such that for any $\vh \in \cH$, $\vj\in \cJ$ satisfying~\eqref{eq:eqn1} and~\eqref{eq:zz1} there exists $\ell_1,\ldots,\ell_r$ such that
$$
(\talpha_0+\talpha_1\ell_1+\ldots + \talpha_r\ell_r)(\beta_0 +\beta_1h_1+\ldots +\beta_eh_e)\equiv \lambda \mod{p},
$$
and 
$$
\talpha_0+\talpha_1\ell_1+\ldots + \talpha_r\ell_r= \alpha_0+\alpha_1h_1+\ldots+\alpha_dh_d.
$$
Let  $\td=r$ and let $\tcA$ denote the generalized arithmetic progression 

$$
\tcA=\left \{ \talpha_0+\talpha_1 \widetilde  h_1+\ldots +\talpha_{\td} \widetilde h_{\td} :~
 \left|\widetilde h_i\right |\le \tCde H, \ i =1, \ldots, \td\right \}.
$$
From construction of $\cH$, for each $a\in \cA$ there exists a unique $\bf{h}\in \cH$ such that 
$$
a=\alpha_0+\alpha_1h_1+\ldots+\alpha_dh_d.
$$
For each $\bf{h}\in \cH$, there exists some $\widetilde{\bf h}$ satisfying 
$$ \left|\widetilde h_i\right|\le \tCde H, \quad  i=1,\ldots, \td,$$
such that 
$$ \alpha_0+\alpha_1h_1+\ldots+\alpha_d h_d=\talpha_0+\talpha_1\widetilde h_1+\ldots +\talpha_{\td} 
\widetilde h_{\td}.$$
The above implies we may choose a set 
$$\widetilde \cH\subseteq [-\tCde H,\tCde H]^{\td},$$
such that the points 
$$\talpha_0+\talpha_1\widetilde h_1+\ldots +\talpha_{\td}\widetilde h_{\td}, \quad \widetilde{\bf{h}}\in \cH_0,$$
are distinct and for each $\bf{h}\in \cH$ satisfying~\eqref{eq:eqn1} and~\eqref{eq:zz1} there exists some $\widetilde{\bf{h}}\in \cH_0$ such that 
$$ \alpha_0+\alpha_1h_1+\ldots+\alpha_d h_d=\talpha_0+\talpha_1\widetilde h_1+\ldots +\talpha_{\td}. 
\widetilde h_{\td}.$$
The above combined with~\eqref{eq:IABK} implies that
$$J_1(\cA,\cB,\lambda)\le I_p(\tcA,\cB,\lambda).$$ 
Hence, recalling~\eqref{eq:IpJp}, we obtain the desired result provided~\eqref{eq:case11} holds.

If we are in the case of~\eqref{eq:case12}, then by Lemma~\ref{lem:linear} there exists integers $a_i$ and $b_i$ satisfying 
$$
\frac{\rho_i}{\rho_d}=\frac{a_i}{b_i}, \quad \gcd(a_i,b_i)=1, \quad a_i, b_i \ll H^{d}, \qquad 1\le i \le d,
$$
where by symmetry we assume $j=1$ and  also use that $\rho_1=1$ in our application of Lemma~\ref{lem:linear}. 
By~\eqref{eq:large p}, provided that $C_d$ is large enough,  
 we see that 
$$
\text{if} \quad a_i,b_i\neq 0 \quad \text{then} \quad a_i,b_i\not \equiv 0 \mod{p}.
$$ 
By~\eqref{eq:zz1} and~\eqref{eq:h123}
$$
h_d-h_d^{*}=\frac{\alpha_1}{\alpha_d}(h_1^{*}-h_1)+\ldots+\frac{\alpha_{d-1}}{\alpha_d}(h_{d-1}^{*}-h_{d-1}),
$$
which combined with the above implies 
\begin{equation}
\label{eq:gdcase2}
h_d\equiv h_d^{*}-a_1\overline{b_1}(h_1-h_1^{*})-\ldots - a_{d-1}\overline{b_{d-1}}(h_{d-1}-h_{d-1}^{*}) \mod{p},
\end{equation}
where $\overline{x}$ denotes the multiplicative inverse of $x$ modulo  $p$. As before, substituting~\eqref{eq:gdcase2} into~\eqref{eq:eqn1}, there exists a generalized arithmetic progression 
$$
\tcA=\{ \talpha_0+\talpha_1\ell_1+\ldots +\talpha_{\td} \ell_{\td} :~ |\ell_i|\le \tCde H, \ i =1, \ldots, \td \},
$$ 
with $\td<d$
such that
$$
J_1(\cA,\cB,\lambda)\le I_p(\tcA,\cB,\lambda),
$$
 and the result follows combining this with~\eqref{eq:IpJp}.

In the case of~\eqref{eq:IpJp1}, we apply a similar argument as before, except with the  lattice 
$$
\cL=\left\{ (n_1,\ldots,n_e)\in \Z^{e} :~  \tau_1 n_1+\ldots+\tau_e n_e=0 \right\},
$$
and convex body 
$$
D=\{ (n_1,\ldots,n_e) :~ |n_i|\le H, \ i =1, \ldots, e\},
$$ 
to obtain 
$$
J_2(\cA,\cB,\lambda)\le I_p(\cA,\tcB,\lambda),
$$ 
for some generalized arithmetic progression $\tcB$ of the form 
$$
\tcB=\{ \tbeta_0+\tbeta_1h_1+\ldots +\tbeta_{\te} h_{\te} :~ |h_i|\le \tCde H, \ i =1, \ldots, \te\},
$$
with $\te<e$.  
Combining this with~\eqref{eq:IpJp1}  we obtain the desired inequality under the assumption~\eqref{eq:large p}. 
To conclude that proof it remains to verify~(ii), about the divisibility of $Z_{d,e}$.

\subsection{Prime divisors of $Z_{d,e}$} 
We now show that $Z_{d,e}$ is divisible by all  primes $p \le   H^{d+e+o(1)}$. Fix some small $\varepsilon>0$ and consider generalised arithmetic progressions of the form~\eqref{eq:Aiter} and~\eqref{eq:Biter}. 


We next use the Dirichlet pigeon-hole principle to show the statement of Lemma~\ref{lem:iter} fails  
for any prime $H^d < p \le  H^{d+e-\varepsilon}$ provided that $H$ is 
large enough. This is sufficient from~\eqref{eq:Z0def} and the fact that $Z_0|Z_{d,e}$, (provided $C_d>1$). 

Indeed, since the value of $Z_{d,e}$ does not depend on the generalised arithmetic progressions $\cA_0$ and $\cB_0$, we can 
choose 
$$
\alpha_0 = 0 \mand  \alpha_i = (2H+1)^{i-1}, \quad i =1, \ldots, d,
$$
and 
$$
\beta_0 = 0 \mand  \beta_i = (2H+1)^{j-1}, \quad j =1, \ldots, e.
$$
Hence  $\cA_0$ and $\cB_0$ are proper and in fact  contain $(2H+1)^d$ and $(2H+1)^e$
distinct residues modulo $p$, respectively.  
Next, there are $ (2H+1)^{d+e}$ products in $\overline{\mathbb{F}}_p$
$$
\(\alpha_0+\alpha_1h_1+\ldots +\alpha_d h_d\) \(\beta_0+\beta_1j_1+\ldots +\beta_e j_e\)  
$$
over all choices of $|h_i|\le H$, $i =1, \ldots, d$ and  
$|j_i|\le  H$, $i =1, \ldots, e$, except for at most $O\(H^{d+e-1}\)$ 
choices for which this product is divisible by $p$. Hence, there exists a non-zero residue class $\lambda_0$ modulo $p$ into which 
at least 
$$ 
(2H+1) ^{d+e} \(1+ O\(H^{-1}\)\)/p \gg H^{\varepsilon}
$$ 
of such products fall, thus giving 
$$
I_p(\cA_0,\cB_0,\lambda_0) \gg H^{\varepsilon}
$$
 contradicting the assumed bound. 

Hence we have $p \mid Z_{d,e}$ for such primes. This also implies that  the assumption~\eqref{eq:large p} holds for any prime $p\nmid Z_{d,e}$. 
 
\section{Proofs of results on factorisation in generalised arithmetic progressions} 

 \subsection{Proof of Theorem~\ref{thm:main1}} 
 We proceed by induction on $d+e$  with base case 
$$d+e=1.$$
In this case, there exists some $\lambda_0\in \F_p$ such that
$$I_{p}(\cA,\cB,\lambda)=|\{ 1\le h\le H \ : \ h=\lambda_0 \}|,$$
 for which there is at most $1$ solution provided $H\le p$. Hence the result 
follows by taking 
$$Z=\prod_{p\le H}p.$$
 
Let $C_*(\ell) $ be sufficiently large depending only on the implied constants in Lemma~\ref{lem:iter}.

  We next set up some notation related to our induction hypothesis. Let $H\gg 1,$ $\ell\ge 2$  and for each pair of positive integers $d,e$ satisfying 
$$e\le d \mand  d+e \le \ell,$$
let $Z_{d,e}$ be as in Lemma~\ref{lem:iter}. Define 
$$
Z_{\ell}=\prod_{\substack{0<e\le d \\ d+e\le \ell \\}}Z_{d,e},
$$
so that $Z_{\ell}$ is $O(H^{1/\gamma_{\ell}})$-smooth and satisfies 
$$
\log{Z_{\ell}}\ll \log{H}\sum_{3\le j \le \ell}\sum_{\substack{0<e\le d \\ d+e=j \\}}H^{j(d+1)(e+1)}\ll H^{(\ell-1)(\ell+2)^2/4} \log H
$$
where we have used 
$$
j(d+1)(e+1) \le j \(\frac{d+e+2}{2}\)^2 \le \ell \frac{(\ell+2)^2}{4}.
$$

We formulate our induction hypothesis as follows. 
There exists a constant  $b_{\ell-1}$ such that for any positive integers $e\le d$ 
satisfying $d+e \le \ell-1$  and  prime
\begin{equation}
\label{eq:p-large}
p \ge  C_*(\ell) H^{\ell-1}
\end{equation} 
not dividing $Z_{\ell-1}$ (which by Lemma~\ref{lem:iter}~(ii) holds for any $p \nmid Z_{\ell-1}$), for any $\lambda \in \F_p$ and generalized arithmetic progressions

\begin{align*}
\cA&=\left \{\alpha_0+ \alpha_1h_1+\ldots+\alpha_dh_d :~ 1\le h_i\le H, \ i =1, \ldots, d\right \},\\
\cB&=\left \{\beta_0+ \beta_1j_1+\ldots+\beta_ej_e :~ 1\le j_i\le H, \ i =1, \ldots, e\right \},
\end{align*}
 we have 
$$
I_p(\cA,\cB,\lambda)\le \exp\left(b_{\ell-1}\log{(|\cA||\cB|)}/\log \log{(|\cA||\cB|)}\right). 
$$
 
Let $e\le d$ be positive integers satisfying 
$$d+e=\ell$$
and $H\gg 1$. 

By Lemma~\ref{lem:iter},  for any prime $p \nmid   Z_\ell$, and thus satisfying
$$
p\gg  H^{d+e}
$$
the following holds.

 Let $\lambda\in \overline\F_p^*$
and $\cA,\cB\subseteq \overline\F_p$  generalised arithmetic progressions
as  in~\eqref{eq:Aiter} and  satisfying \eqref{eq:Biter}
with $d,e\ge 2$ and
$$\alpha_1,\ldots,\alpha_d,\beta_1,\ldots,\beta_e \in \overline \F_p^{*}.
$$  
 There exists a constant $\tCde$  depending only on $d,e$,  integers $\td$ and $\te$ satisfying 
$$
\td \le d, \qquad \te \le e, \qquad  \td+\te<d+e,
$$ 
 generalised arithmetic progressions $\tcA,\tcB$ of the form
\begin{align*}
& \tcA=\{ \talpha_0+\talpha_1h_1+\ldots +\talpha_{\td} h_{\td} :~ |h_i|\le \tCde H, \  i =1, \ldots, \td\},\\
& \tcB=\{ \tbeta_0+ \tbeta_1j_1+\ldots + \tbeta_{\te} j_{\te} :~ |j_i|\le \tCde H, \ i =1, \ldots, \te \},
\end{align*}
with 
$$\talpha_1 \ldots,\talpha_d, \tbeta_1 \ldots \tbeta_e \in \overline \F_p^{*},$$
and some $\mu \in \overline\F_p^*$ such that 
$$
I_p(\cA,\cB,\lambda)\le \exp\left(B_d \log{H}/\log\log{H}\right)I_p(\tcA,\tcB,\mu ).
$$ 
If $p \nmid Z_{\ell}$ then  we obviously have
$p \nmid Z_{\ell-1}$  (since $Z_{\ell-1} \mid Z_\ell$) and also $p$ satisfies~\eqref{eq:p-large}.
Therefore, by our induction hypothesis
(where we can also assume that $\Cde \ge C_*(d+e-1)$) 
$$
I_p(\cA,\cB,\lambda) \ll \exp\left((B_d+b_{\ell-1}) \log{H}/\log\log{H}\right).
$$
The result now follows by taking 
$$b_\ell =\max_{d \le \ell} B_d+b_{\ell-1}, \quad Z=Z_{\ell}$$
and noting  
 that $Z_\ell$ is $O(H^{1/\gamma_{\ell+1}})$-smooth and satisfies 
$$
\log Z_\ell\ll H^{\ell(\ell+2)^2/4} \log H.
$$

 \section{Proofs of results towards  the Erd\H{o}s--Szemer{\'e}di conjecture} 

\subsection{Proof of Theorem~\ref{thm:main24}}  
The celebrated theorem of Freiman~\cite{Frei}  states that if $\cA\subseteq \Z$ is a finite set satisfying 
$$|\cA+\cA|\le K|\cA|,$$  
then there exist  constants $b(K)$ and $d(K)$ depending only on $K$, and some generalised arithmetic progression 
$\cB$ of rank $d(K)$ and size 
$$|\cB|\le b(K) |\cA|,$$
such that 
$$\cA\subseteq \cB.$$

The theorem of Freiman~\cite{Frei}  has gone through a number of improvements and generalisations
to sets from arbitrary abelian groups. 

A version of this result convenient for our application is due to  Cwalina and  Schoen~\cite[Theorem~4]{CS}, which states that we may take $\cB$ proper,
$$
b(K) \le   \exp\(c K^4(\log{K+2})\)  \mand 
d(K) \le 2K,
$$  
for some absolute constant $c$ (note the additive group of $\F_p$ has no proper subgroups, 
so only the first alternative of~\cite[Theorem~4]{CS} applies).

Thus, using Corollaries~\ref{cor:all-primes} and~\ref{cor:main22} with $H = |\cA|$, $d = e = d(K)$, 
$$
\delta=\gamma_{2d(K)+1}=\frac{1}{(44K+26)2^{12K+8}}  
\quad \text{and}\quad  c_0(K) = C_0(2K), 
$$  
where  $C_0(d)$ is as in Corollary~\ref{cor:all-primes} (which we can assume to be 
monotonically increasing with respect to both $d$ and $e$), 
we obtain that for each $\lambda \in \F_q^{*}$, the number of solutions to each of the equations 
$$
a_1a_2 =  \lambda ,  \qquad  a_1^{-1}+ a_2^{-1}= \lambda,
 \qquad  a_1^{2}+ a_2^{2}= \lambda ,
$$
over $\F_p$
with variables $a_1,a_2\in \cA$ is $|\cA|^{o(1)}$ since we assume that $K$ is fixed, from which the desired result follows.

\subsection{Proof of Theorem~\ref{thm:main24-AA}}  

We follow the proof of Theorem~\ref{thm:main24} however apply Corollary~\ref{cor:almost-all-primes} instead of Corollary~\ref{cor:all-primes}.

\section*{Acknowledgement}

The authors would like to thank Giorgis Petridis for pointing out that  Theorem~\ref{thm:main24}
is a finite field analogue of a result of Elekes and Ruzsa~\cite{ER}.
The authors are also grateful Misha Rudnev for many useful comments and 
queries, which helped to discover a gap in the initial version. 

During this work, B.K. was  supported  by Australian Research Council Grant~DP160100932,  Academy of Finland Grant~319180 and the Max Planck Institute for Mathematics, 
J.M.  
by Australian Research Council Grant~~DP180100201, by NSERC and by the Max Planck Institute for Mathematics,  and I.S.  
by Australian Research Council Grants DP170100786 and DP180100201.

\end{document}